%% file: main.tex
\documentclass{article} %
\usepackage{amssymb,amsfonts,amsmath,amsthm}
\usepackage{bbm}
\usepackage{enumitem}
\usepackage{algorithm}
\usepackage[noEnd]{algpseudocodex}
\usepackage{mathdots}
\usepackage{mathtools}
\usepackage{fullpage}
\definecolor{ao(english)}{rgb}{0.0, 0.5, 0.0}

\usepackage[colorlinks, citecolor = {ao(english)}, linkcolor = {ao(english)}]{hyperref} 
\usepackage[nameinlink]{cleveref}
\usepackage{aux}

\DeclareMathOperator{\adj}{adj}
\DeclareMathOperator{\comp}{C}
\DeclareMathOperator{\rank}{\mathrm{rank}}

\algnewcommand{\LineComment}[1]{\Statex \(\triangleright\) #1}

\usepackage{graphicx}
\usepackage{mathdots}
\usepackage{amsopn}

\usepackage{dsfont}
\let\mathbb=\mathds
\newtheoremstyle{boldtitle}
{3pt}{3pt}      %
{}              %
{}              %
{\bfseries}     %
{.}             %
{ }             %
{}              %

\theoremstyle{boldtitle}
\newtheorem{thm}{Theorem}
\crefname{thm}{Theorem}{Theorems}
\newlist{thmenum}{enumerate}{1} %
\setlist[thmenum]{label=\roman*), ref=\thethm~\roman*)}
\crefalias{thmenumi}{thm} 
\newtheorem{prop}{Proposition}
\crefname{prop}{Proposition}{Propositions}
\newlist{propenum}{enumerate}{1} %
\setlist[propenum]{label=\roman*), ref=\theprop~\roman*)}
\crefalias{propenumi}{prop}

\newtheorem{lem}{Lemma}
\crefname{lem}{Lemma}{Lemmas}
\newlist{lemenum}{enumerate}{1} %
\setlist[lemenum]{label=\roman*), ref=\thelem~\roman*)}
\crefalias{lemenumi}{lem}

\newtheorem{cor}{Corollary}
\crefname{cor}{Corollary}{Corollaries}
\newlist{corenum}{enumerate}{1} %
\setlist[corenum]{label=\roman*), ref=\thecor~\roman*)}
\crefalias{corenumi}{cor} 

\newtheorem{rem}{Remark}
\crefname{rem}{Remark}{Remarks}
\newlist{remenum}{enumerate}{1} %
\setlist[remenum]{label=\roman*), ref=\therem~\roman*)}
\crefalias{remenumi}{rem} 

\newtheorem{ex}{Example}
\crefname{ex}{Example}{Examples}
\newtheorem{ass}{Assumption}
\crefname{ass}{Assumption}{Assumption}

\crefname{conj}{Conjecture}{Conjectures}

\crefname{defn}{Definition}{Definitions}
\newlist{defnenum}{enumerate}{1} %
\setlist[defnenum]{label=\roman*., ref=\thedefn~(\roman*.)}
\crefalias{defnenumi}{defn}

\crefname{definition}{Definition}{Definitions}
\usepackage[toc,page]{appendix}
\crefname{appendix}{Appendix}{Appendices}
\Crefname{appendix}{Appendix}{Appendices}

\crefname{algorithm}{Algorithm}{Algorithms}
\crefformat{equation}{\textup{#2(#1)#3}}
\crefrangeformat{equation}{\textup{#3(#1)#4--#5(#2)#6}}
\crefmultiformat{equation}{\textup{#2(#1)#3}}{ and \textup{#2(#1)#3}}
{, \textup{#2(#1)#3}}{, and \textup{#2(#1)#3}}
\crefrangemultiformat{equation}{\textup{#3(#1)#4--#5(#2)#6}}%
{ and \textup{#3(#1)#4--#5(#2)#6}}{, \textup{#3(#1)#4--#5(#2)#6}}{, and \textup{#3(#1)#4--#5(#2)#6}}

\Crefformat{equation}{#2Equation~\textup{(#1)}#3}
\Crefrangeformat{equation}{Equations~\textup{#3(#1)#4--#5(#2)#6}}
\Crefmultiformat{equation}{Equations~\textup{#2(#1)#3}}{ and \textup{#2(#1)#3}}
{, \textup{#2(#1)#3}}{, and \textup{#2(#1)#3}}
\Crefrangemultiformat{equation}{Equations~\textup{#3(#1)#4--#5(#2)#6}}%
{ and \textup{#3(#1)#4--#5(#2)#6}}{, \textup{#3(#1)#4--#5(#2)#6}}{, and \textup{#3(#1)#4--#5(#2)#6}}

\crefdefaultlabelformat{#2\textup{#1}#3}

\usepackage[ backend=biber, style=numeric, eprint=true,doi=false,url=false,giveninits=true,style=numeric-comp]{biblatex}
\AtEveryBibitem{\clearfield{issn}}
\title{Inversion of the Multiplicative Matrix Compound Operator \thanks{The work of Debojyoti Dey and Christian Grussler was supported by the Israel Science Foundation (grant no. 2406/22) and the Bernard M. Gordon Center for Systems Engineering at the Technion -- IIT. The work of Ron Ofir was supported in part by the Air Force Office of Scientific Research, under award numbers FA9550-23-1-0175 and FA9550-25-1-0223, and by the Viterbi Fellowship, Technion -- IIT.}}

\author{Debojyoti Dey,  \; \thanks{D. Dey is with the Stephen B. Klein Faculty of Aerospace Engineering, Technion --- Israel Institute of Technology, 3200003 Haifa, Israel
		{\tt debojyot@campus.technion.ac.il}}  Ron Ofir \thanks{R. Ofir is with the Department of Electrical \& Computer Engineering, Yale University, 06520 New Haven, CT, United States 
		{\tt ron.ofir@yale.edu}}, \; Christian Grussler\thanks{C. Grussler is with the Stephen B. Klein Faculty of Aerospace Engineering, Technion --- Israel Institute of Technology, 3200003 Haifa, Israel
		{\tt cgrussler@technion.ac.il}}} 

\begin{document}

\maketitle

\begin{abstract}
	We study the problem of determining a matrix whose $k$th multiplicative compound, with~$k > 1$, is a prescribed matrix~$M$. The cardinality of the set of matrices whose $k$th multiplicative compound equals~$M$ is characterized in terms of $\rank(M)$. On the one hand, if $\rank(M)\le 1$, it is shown that there exist infinitely many such matrices for which a complete characterization is determined. On the other hand, if $\rank(M)>1$, then there exists a unique matrix --- up to an overall sign --- whose compound is~$M$. An algorithm for finding a matrix whose compound equals~$M$ is detailed, and its time complexity is analyzed. %
\end{abstract}
\input{intro}

\input{prelim}

\input{main_result}

\input{proof_main}

\input{conclusion}

\printbibliography

\begin{appendices}
	\input{a1}
\end{appendices}

\end{document}

%% file: intro.tex
\section{Introduction}
\label{sec:intro}
Compound matrices are a classical object in linear and multilinear algebra, appearing in a range of applications, including geometry~\cite{Gantmacher1960VolI}, graph theory~\cite{Fiedler1998AdditiveCompGraphs,Bapat2004GraphEnergy}, algebraic topology~\cite{Lew2024Laplacian}, numerical analysis~\cite{Ng1985CompoundMethod,Allen2002NumericalExteriorAlgebra}, and electric network theory~\cite{Bryant1963compound}. More recently, they have proven useful in dynamical systems and control theory~\cite{BarShalom2023Tutorial}. Following the seminal works of Schwarz~\cite{Schwarz1970} and Li and Muldowney~\cite{Muldowney1990compound,Li1995}, which introduced compound matrices into this setting, there has been growing interest in their use for studying linear and nonlinear systems. System operators with so-called variation-diminishing and -bounding properties have been studied in \cite{Grussler2022variation,grussler2021internally,grussler2026efficient,grussler2024system} with applications to model order reduction \cite{grussler2020balanced}, frequency domain analysis \cite{grussler2023monotonicity}, bounds on self-sustained oscillations \cite{tong2026unimodal}, as well as sparse optimization \cite{marmary2025tractabledownfallbasispursuit}. Autonomous system analysis brought forward new results on the global stability of systems with multiple equilibrium points~\cite{Wu2022kcont, Angeli2021nonosc}, with applications ranging from power systems~\cite{Cecilia2026genlyap}, to opinion dynamics~\cite{Ofir2024Luriekcont} and neural networks~\cite{Richardson2025lurie}. All of these results follow a common paradigm: a classical condition formulated in terms of a matrix $X \in \mathds{R}^{n \times m}$ (e.g., the Jacobian of the vector field of an autonomous system or the system operators associated with an input-output system) is lifted to a condition on its $k$th order minors or, equivalently, on the \emph{$k$th multiplicative compound matrix} $C_k(X) \in \mathds{R}^{\binom{n}{k} \times \binom{m}{k}}$ (see~\Cref{sec:prelims} for a formal definition). 

As a concrete example, a linear time-invariant system $$x(t+1) = Ax(t)$$ is called $k$-positive if all entries of the~$k$th multiplicative compound of~$A$ have the same sign~\cite{Alseidi2021DTkpos}, which is an extension of the well-known class of \emph{internally positive systems} \cite{farina2011positive,luenberger1979introduction,ohta1984reachability}, where $A$ is a nonnegative matrix (i.e., all its $1$st order minors share the same sign). However, as these properties are coordinate dependent, the question arises whether for given $A$ there exists a coordinate transformation $P \in \mathds{R}^{n \times n}$ such that $y(t+1) := P^{-1}x(t+1) = P^{-1}APy(t)$ defines a $k$-positive system, or equivalently, 
\begin{equation}
    \exists P, N \in \mathds{R}^{n \times n}: AP = P N, \label{eq:PN}
\end{equation}
where $C_k(N)$ is a nonnegative/nonpositive matrix. In case of $k=1$, this question is known as (or part of) the so-called \emph{positive realization problem} \cite{benvenuti2004tutorial,farina2011positive,ohta1984reachability}, and equivalent to the \emph{nonnegative inverse eigenvalue problem (NIEP)} \cite{loewy2017necessary,egelston2004nonnegative,cronin2018diagonalizable,grussler2022similarity}. Problems with $k > 1$ have also been studied, for example, in \cite[Theorem $3.10$]{grussler2021internally} to find an internally Hankel $k$-positive realization of an LTI system, and \cite{Weiss2021IsKpos}, which studies a related problem in the continuous-time setting and for a restricted class of coordinate transformations.

While no closed-form solution to \label{eq:PN} is currently known, heuristic approaches solve \eqref{eq:PN} for $k=1$ through alternating iterative optimization \cite{orsi2006numerical}. In an attempt to derive similar heuristic approaches for the cases of $k > 1$, one may rewrite \eqref{eq:PN} by the \emph{Cauchy-Binet formula}\cite[Subsection $0.8.7$]{hornMatrix} as
\begin{equation*}
    \exists P, N \in \mathds{R}^{n \times n}: C_k(A)C_k(P) = C_k(P) C_k(N), \label{eq:PN_k}
\end{equation*}
and, therefore, take advantage of the NIEP by finding $Q,M \in \mathds{R}^{\binom{n}{k} \times \binom{n}{k}}$ such that
\begin{equation*}
    \comp_k(A)Q = QM
\end{equation*}\
with nonnegative/nonpositive $M$, and complementing this approach with projections of $Q$ and $M$ onto the manifold $$\{\comp_k(X) \in \mathds{R}^{\binom{n}{k} \times \binom{n}{k}}: X \in \mathds{R}^{n \times n}\}.$$ While we leave a full algorithm to future work, the first immediate question is to find the $k$th \textit{inverse compound} $N$, such that $M=\comp_k(N)$ (similarly finding $P$ such that $\comp_k(P)=Q$). Despite the long history of compound matrices, this inverse problem has not been considered before, to the best of the authors' knowledge. Given a matrix $M$ of appropriate dimensions, two fundamental questions arise:
(i) What are the necessary and sufficient conditions for the existence of a matrix $A$ whose~$k$th multiplicative compound is equal to~$M$? 
(ii) Assuming such a matrix exists, how can one characterize and compute all matrices $A$ satisfying this relation?

In this paper, we address the second question. We provide a complete characterization of the preimage of the $k$th multiplicative compound operator, showing in particular that the preimage is unique --- up to an overall sign --- under certain conditions, and that otherwise it contains infinitely many matrices. We then develop a constructive algorithm for recovering all matrices $A$ whose $k$th compound equals $M$, and show that the computational complexity of the algorithm is polynomial in the dimension of the matrix. %
A MATLAB implementation of the proposed algorithms is available at \url{https://github.com/debojyoti23/inverse-compound}.

%% file: prelim.tex
\section{Preliminaries}
\label{sec:prelims}
In this section, we introduce the basic notations and concepts that are useful for our subsequent discussions and derivations.

\subsection{Notations}
We use $\bR$ for the set of real numbers, $\bR_{>0}$ for positive real numbers, and $\bZ$ for the set of integers.
For $r,s\in\bZ$ we write $(r:s)=\{r,r+1,\ldots,s\}$ where $r\leq s$ and if $r>s$, the notation denotes the empty set. For the special case of $r=1$ and a positive integer $n$, we also use $[n]:=(1:n)$. The cardinality of a set $\cS$ is denoted by $|\cS|$. For a tuple $I\subset (1:n)$, we denote its complement by $I^c$ such that $I\cup I^c = (1:n)$.
A matrix $X = (x_{ij}) \in\bR^{n\times m}$ is a rectangular matrix over the real field with $n$ rows and $m$ columns.
For a set of vectors $v_1,\ldots,v_n\in\bR^n$, we use $v_{ji}$ to denote the $j$th element of $v_i$.
The exponentional $\exp(\cdot)$ and logarithm $\log(\cdot)$ are applied element-wise on vectors, i.e., for $v\in\bR^n$, $\exp(v)_i=\exp(v_i)$ and $\log(v)_i=\log(v_i)$.
We use $\diag(X)$ to denote the vector of the elements of the main diagonal of $X$, whereas,
$\diag(u)$ for $u\in\bR^n$ denotes a diagonal matrix with $u$ being the main diagonal.
The identity matrix in $\bR^{n\times n}$ is denoted by $I_n$, or simply by $I$ where the dimension is obvious.
The submatrix of $X \in \mathds{R}^{n \times m}$ with row-indices $I\subset (1:n)$ and column-indices $J\subset (1:m)$ is denoted as $X_{I,J}$. In particular, $X_{:,J} := X_{(1:n),J}$ and similarly, $X_{I,:} :=  X_{I,(1:m)}$. We denote the ordered set of increasing $k$-tuples by,
\[
\cI_{n,k} := \{(v_1,\ldots,v_k) \cond 1\leq v_1<\cdots<v_k\leq n, v_i\in\bN \, \forall i\}
\]
where the $i$th element of this set is with respect to the \textit{lexicographic ordering.}
The $k$th \textit{multiplicative compound} of $X \in \mathds{R}^{n \times m}$, denoted by $\comp_k(X)\in \bR^{\binom{n}{k}\times\binom{m}{k}}$, is the matrix whose~$(i,j)$th element is $\det (X_{I,J})$ where $I$ is the $i$th element in $\cI_{n,k}$ and $J$ is the $j$th element in $\cI_{m,k}$.
For example, let $X\in\bR^{3\times 3}$, then
\[
\comp_k(X) = 
\begin{pmatrix}
    \det (X_{\{1,2\},\{1,2\}}) & \det (X_{\{1,2\},\{1,3\}}) & \det (X_{\{1,2\},\{2,3\}})\\
    \det (X_{\{1,3\},\{1,2\}}) & \det (X_{\{1,3\},\{1,3\}}) & \det (X_{\{1,3\},\{2,3\}})\\
    \det (X_{\{2,3\},\{1,2\}}) & \det (X_{\{2,3\},\{1,3\}}) & \det (X_{\{2,3\},\{2,3\}})
\end{pmatrix}
\]
We, further, define $\comp_{0}(X):=1$ and note that $\comp_n(X)=\det (X)$ if $X\in \bR^{n\times n}$. The operator is called \textit{multiplicative} because for two matrices $A\in\bR^{n\times p}$ and $B\in\bR^{p\times m}$, the so-called \emph{Cauchy-Binet formula} \cite[Subsection $0.8.7$]{hornMatrix} implies
\begin{equation}
    \label{eq:cauchybinet}
    \comp_k(AB) = \comp_k(A)\comp_k(B) \ \ (k\leq \min\{m,n,p\})
\end{equation}
For a compound matrix $M=\comp_k(X)$, we use $M(I, J)$ to denote the $(i,j)$th element of $M$, where $I$ and $J$ are the $i$th and $j$th element in $\cI_{n,k}$ respectively. The following \cref{thm:compoundProperty} summarizes the properties (see \cite[Subsection $0.8.1$]{hornMatrix}, \cite[Subsection $0.iv$]{karlinTP}) of a compound matrix, which will be used in our further discussion.
\begin{lem}
\label{thm:compoundProperty}
Let $A\in \bR^{n\times m}$.
\begin{enumerate}
    \item If $n=m$, then~$A$ is non-singular if and only if~$\comp_k(A)$ is non-singular, in which case $\comp_k(A^{-1})=\comp_k(A)^{-1}$.
    \item For any~$t \in \bR$, $\comp_k(tA)=t^k \comp_k(A)$.
    \item If $n=m$, then the eigenvalues of $\comp_k(A)$ are given by
    $\prod_{i\in I}\lambda_i(A)$, $I\in\cI_{n,k}$.
    The same holds for singular values and arbitrary~$n,m$.
    \item If $\rank(A)=r$, then $\rank(\comp_r(A))=1$. 
    \item For $\rank(A)<k\leq \min\{n,m\}$, $\comp_k(A)=0$.
    \item If $A$ is diagonalizable, and~$u_1,\dots,u_n \in \bR^n$ is a linearly independent set of eigenvectors of~$A$, then $\{\comp_k(\begin{bmatrix}u_1 & \cdots & u_k\end{bmatrix}) : I \in \cI_{n,k}\}$ is a linearly independent set of $\binom{n}{k}$ dimensional eigenvectors of~$\comp_k(A)$. In particular,~$\comp_k(A)$ is diagonalizable.
    \item If $n=m$ and $n>1$ then the eigenvalue property in item $3$ above implies $\det (\comp_k(A)) = (\det (A))^{\binom{n-1}{k-1}}$. This is also known as the \textit{Sylvester-Franke theorem}.
\end{enumerate}
\end{lem}

An immediate consequence of item 3 of \cref{thm:compoundProperty} is the following identity between the rank of a matrix and its compound.
\begin{lem}\label{thm:compound_rank}
    Let~$A \in \bR^{n \times m}$. For every~$1 \le k \le \rank(A)$,
    \[
        \rank(\comp_k(A)) = \binom{\rank(A)}{k}.
    \]
\end{lem}
The following result from \cite[Subsection 0.8.2]{hornMatrix} gives a closed-form formula to compute the preimage $A$ of a given adjugate matrix $\adj(A)$.
\begin{lem}
\label{thm:adj2}
    Let $A\in\bR^{n\times n}$ and $M=\adj(A)$, then $\adj (M)=(\det (M)) ^\frac{n-2}{n-1} A$.
\end{lem}

Since $\adj(A)$ is computed by $(n-1)$-th order minors of $A$, we further get the following reciprocal relation between the adjugate matrix $\adj(A)$ of $A$ and $\comp_{n-1}(A)$. 
\begin{lem}
\label{thm:adjcomp_bijection}
For $A\in \bR^{n\times n}$, $n > 1$, the following holds:
    \[\adj(\adj(A)) = \comp_{n-1}(\comp_{n-1}(A))\]
\end{lem}
A proof of \cref{thm:adjcomp_bijection} can be found in \cref{app:proofs}.

\subsection{Exterior Algebra and Compound}
\label{sec:extAlgebra}
The $k$th exterior power $V^{(k)}$ of an $n$-dimensional vector space $V$ is a vector space of dimension $\binom{n}{k}$ spanned by vectors of the form,
\[
z = u_1\wedge u_2\wedge \cdots \wedge u_k \quad (u_i\in V)
\]
where the product above is called an exterior or wedge product~\cite{winitzki2023linear}. 
In this paper, we focus on $V = \bR^n$. Given the \emph{canonical basis} $\{e_1,\ldots,e_n\}$ of $\bR^n$, the set
\[
\{e_{i_1} \wedge \cdots \wedge e_{i_k} : (i_1,\dots,i_k) \in \cI_{n,k}\}
\]
is a basis of the $k$th exterior power $V^{(k)}$, where each basis vector corresponds to a $k$-dimensional oriented subspace. Let $U = [u_1, u_2,\ldots, u_n]\in \bR^{n\times n}$, $u_i\in \bR^n$, then the $i$th coordinate of $z$, as defined above, will be given by the \textit{Grassmann notation} as
\begin{align}
\label{eq:grassman}
z_{I = (i_1,\ldots,i_k)} = \det (U_{\{i_1,\ldots,i_k\},1:k})
\end{align}
where $I$ is the $i$th element of $\cI_{n,k}$. Thus, $\comp_k(U)\in\bR^{\binom{n}{k}\times \binom{n}{k}}$ is a matrix with each column representing the coordinates of $k$-wedge $u_{i_1}\wedge\ldots \wedge u_{i_k}$ for all $(i_1,\ldots,i_k)\in \cI_{n,k}$, lexicographically ordered.
The two primary properties \cite[Section $0.iii$]{karlinTP} of the wedge product are:
\begin{enumerate}
    \item multi-linearity, i.e., $(\alpha u_1 + \beta u_2)\wedge u_3 = \alpha(u_1\wedge u_3) + \beta (u_2\wedge u_3)$;
    \item anti-symmetry, i.e., $u_1\wedge u_2 = - u_2\wedge u_1$, giving us $u\wedge u = 0$ for any $u\in \bR^n$.
\end{enumerate}
By the Cauchy-Binet formula, $\comp_k(A)$ can also be seen as a coordinate representation of 
\begin{equation*}
    \comp_k(A)(x_1\wedge\cdots\wedge x_k) = Ax_1 \wedge \cdots \wedge Ax_k
\end{equation*}
for all $x_i\in\bR^n$ (see~\cite[Section~0.iii]{karlinTP}). Wedge products are also closely related to linear independence, since $u_1 \wedge \cdots \wedge u_k = 0$ if and only if $u_1,\dots,u_k \in \bR^n$ are linearly dependent.

A vector~$z \in V^{(k)}$ is called~\emph{decomposable} if it can be written as the wedge product of~$k$ vectors in~$V$, i.e. if~$z = u_1 \wedge \cdots \wedge u_k$~\cite{Harris1992}. A decomposable vector does not have a unique decomposition, but if a vector is non-zero and decomposable, then all decompositions span the same $k$-dimensional subspace in~$V$. To be precise, suppose~$z \in V^{(k)}$ is decomposable and non-zero, then if~$u_1,\dots,u_k,v_1,\dots,v_k \in V$ are such that
\begin{equation*}
    z = u_1 \wedge \cdots \wedge u_k = v_1 \wedge \cdots \wedge v_k,
\end{equation*}
then
\begin{equation*}
    \spanop\{u_1,\dots,u_k\} = \spanop\{v_1,\dots,v_k\}.
\end{equation*}
This follows from the fact that~$x \wedge z = 0$ if and only if~$x$ is in the span of any decomposition of~$z$.

%% file: main_result.tex
\section{Main Results}
The main objective of this study is to compute the preimage of the multiplicative compound. To be precise, for fixed~$n,m$ and $k\le\min\{n,m\}$, we call a matrix~$M \in \bR^{\binom{n}{k} \times \binom{m}{k}}$ \emph{compound-decomposable} if there exists a matrix~$A \in \bR^{n \times m}$ such that~$M = \comp_k(A)$. We also call $A$ an \emph{inverse-compound} of $M$. Our goal is to find an inverse-compound of a given compound-decomposable matrix. We begin by posing the question:
\begin{center} \textit{Does a compound-decomposable matrix always have a unique inverse-compound matrix?} 
\end{center}
If~$k$ is even, then by item 2 of \cref{thm:compoundProperty}, $C_k(-A) = (-1)^{k} C_k(A) = C_k(A)$, i.e., for even~$k$ uniqueness is at most up to an overall sign. The following \cref{prop:compound_nonunique} shows that, in some cases, there are infinitely many matrices that all have the same compound.
\begin{prop}\label{prop:compound_nonunique}
    Let~$A,B \in \bR^{n \times m}$, and $r = \rank(A)$. Then, the following hold:
    \begin{enumerate}
        \item if $k > r$, then $\comp_k(A) = \comp_k(B) \iff \rank(B) < k$.
        \item if $k = r$ and $A = U V^\top$, with~$U\in \bR^{n\times r},V \in \bR^{m \times r}$, be a rank-revealing factorization of~$A$, then~$\comp_k(A) = \comp_k(B)$ if and only if
        \begin{equation}\label{eq:rank_compound_preimage}
            B \in \{UTV^\top : T \in \bR^{r \times r} \text{ with } \det(T) = 1\}.
        \end{equation}
    \end{enumerate}
\end{prop}
\begin{rem}
    \cref{prop:compound_nonunique} can also be stated in terms of compound-decomposable matrices as follows. Let~$M \in \bR^{\binom{n}{k} \times \binom{m}{k}}$ be compound-decomposable. If~$\rank(M) = 0$, then~$\comp_k(A) = M$ if and only if~$\rank(A) < k$. If~$\rank(M) = 1$, then there exist infinitely many matrices whose~$k$th multiplicative compound equals~$M$.
\end{rem}
\begin{rem}
    When~$k=1$,~$C_k(A) = A$ so clearly~$C_k(A) = C_k(B) \iff A = B$. This is consistent with \cref{prop:compound_nonunique}. If~$1 = k > r$, then~$A = 0$ and~$\rank(B) < k$ similarly implies that~$B = 0$. If~$1 = k = r$, the set in~\eqref{eq:rank_compound_preimage} is singleton, since~$T \in \bR^{1 \times 1} = \bR$ with~$\det(T) = 1$ implies that~$T=1$.
\end{rem}
\begin{proof}
    If~$k > r$, then all minors of order~$k$ of~$A$ are zero, so~$C_k(A) = 0$. The same is also true for every~$B \in \bR^{n \times m}$ with~$\rank(B) < k$.
    
    Suppose now that~$k = r$. We first show that for every rank-revealing factorization of~$A$, the set in~\eqref{eq:rank_compound_preimage} remains the same. Let~$A = U V^\top = \tilde U \tilde V^\top$, with~$U,\tilde U \in \bR^{n\times r}$ and $V ,\tilde V \in \bR^{m \times r}$ full column rank matrices, be two rank-revealing factorizations of~$A$. Theorem 2 of~\cite{Piziak1999FullRank} asserts that there exists an invertible matrix~$R \in \bR^{r \times r}$ such that~$\tilde U = UR$ and~$\tilde V^\top = R^{-1} V^\top$. Therefore,
    \begin{align*}
        \{\tilde U \tilde T \tilde V^\top : \tilde T \in \bR^{r \times r} \text{ with } \det(\tilde T) = 1\} &= \{UR\tilde TR^{-1}V^\top : \tilde T \in \bR^{r \times r} \text{ with } \det(\tilde T) = 1\} \\
            &= \{UTV^\top : T \in \bR^{r \times r} \text{ with } \det(T) = 1\},
    \end{align*}
    where the second equality uses the fact that~$\det(R\tilde TR^{-1}) = \det(R)\det(R^{-1})\det(\tilde T) = \det(\tilde T)$ for every~$\tilde T \in \bR^{r \times r}$. Hence, it remains to show that~$B \in \bR^{n \times m}$ is such that~$\comp_r(B) = \comp_r(A)$ if and only if
    \[
        B \in \cS = \{UTV^\top : T \in \bR^{r \times r} \text{ with } \det(T) = 1\}.
    \]
    Suppose~$B \in \cS$, then by the multiplicativity of the compound
    \[
        \comp_r(B) = \comp_r(UTV^\top) = \comp_r(U)\det(T)\comp_r(V^\top) = \comp_r(UV^\top) = \comp_r(A),
    \]
    where the second equality uses the fact that~$\comp_r(T) = \det(T)$ since~$T \in \mathds{R}^{r \times r}$.

    Conversely, let~$\comp_r(B) = \comp_r(A)$. By \cref{thm:compound_rank},~$\rank(B)$ is such that~$\binom{\rank(B)}{r}=\binom{\rank(B)}{\rank(A)}=1$, i.e., $\rank(B) = \rank(A) = r$. Let~$B = U_BV_B^\top$, where~$U_B\in \bR^{n\times r},V_B \in \bR^{m \times r}$ have full column rank, be a rank-revealing factorization of~$B$. Note that~$\comp_r(U),\comp_r(U_B),\comp_r(V),\comp_r(V_B)$ are all column vectors. Furthermore, multiplicativity of the compound implies that
    \begin{equation}
        \comp_r(A) = \comp_r(U) \comp_r(V)^\top = \comp_r(U_B) \comp_r(V_B)^\top.
    \end{equation}
    So the above are rank-revealing factorizations of~$\comp_r(A)$.
    Since~$\comp_r(A)$ has rank 1, its rank-revealing factorization is unique up to scaling, and so
    \[
        \comp_r(U_B) = c \comp_r(U) \quad \text{and} \quad \comp_r(V_B) = c^{-1}\comp_r(V),
    \]
    for some non-zero~$c \in \bR$.
    The vector~$\comp_r(U) \in \bR^{\binom{n}{r}}$ is a non-zero decomposable column vector, so $\spanop(U) = \spanop(U_B)$ and hence, there exists an invertible matrix~$T_1 \in \bR^{r \times r}$ such that
    \[
        U_B = UT_1.
    \]
    However, since~$\comp_r(U_B) = c\comp_r(U)$, it follows that~$\det(T_1) = c$. Applying the same arguments to~$V_B$ implies that~$V_B = V T_2$ for some invertible~$T_2 \in \bR^{r \times r}$ with~$\det(T_2) = c^{-1}$ and we conclude that
    \[
        B = U_B V_B^\top = U T_1T_2^\top V^\top.
    \]
    Since~$\det(T_1T_2^\top) = 1$, it follows that~$B \in \cS$, which completes the proof.
\end{proof}
The following illustrates \cref{prop:compound_nonunique}. 
\begin{ex}
Consider the following rank $2$ matrices:
\begin{align*}
A =
\begin{pmatrix*}[r]
    1 & 0 & 1\\
    0 & 1 & 0\\
    0 & 1 & 0
\end{pmatrix*}
\ \text{    and    }\ 
B =
\begin{pmatrix*}[r]
    1 & 1 & 1\\
    0 & 1 & 0\\
    0 & 1 & 0
\end{pmatrix*}
\end{align*}
For $k=\rank(A)=\rank(B)$, it holds that
\begin{equation}
    \comp_2(A) = \comp_2(B) = \begin{pmatrix*}[r]
        1 & 0 & -1\\
        1 & 0 & -1\\
        0 & 0 & \phantom{-}0
    \end{pmatrix*},
\end{equation}
which is why \Cref{prop:compound_nonunique} implies that for any rank-revealing factorization $A=UV^\top$, it must hold that $B=UTV^\top$ for some $T$ with $\det(T)=1$.
Indeed, for this example, one may choose
\begin{align*}
U = 
\begin{pmatrix*}
    1 & 0\\
    0 & 1\\
    0 & 1
\end{pmatrix*},
\quad
V =
\begin{pmatrix*}
    1 & 0\\
    0 & 1\\
    1 & 0
\end{pmatrix*},
\quad
T = 
\begin{pmatrix*}
    1 & 1\\
    0 & 1
\end{pmatrix*}.
\end{align*}
In contrast, for $k = 3 > \rank(A)$, $\comp_3(A) = \det(A) = 0$, which is trivially fulfilled for all matrices with rank at most $2$.
\end{ex}

We continue by turning our attention to finding an inverse-compound~$A$ for a prescribed compound-decomposable~$M$. Consider first the special case where~$k = n-1$ and~$M$ is invertible. Then,~$A$ is also invertible by item 1 of~\cref{thm:compoundProperty}, and a closed-form expression for all inverse-compounds can be derived by
combining item $5$ of \cref{thm:compoundProperty}, \cref{thm:adj2}, and \cref{thm:adjcomp_bijection} as follows.
\begin{cor}
\label{thm:closedformInverse}
    Let~$M \in \bR^{n \times n}$,~$n > 1$, be compound-decomposable and $\det(M) \neq 0$, and 
    \[
    B := \det (M)^{-\frac{n-2}{n-1}}\comp_{n-1}(M) \in \bR^{n \times n }.
    \]
    Then~$A \in \bR^{n \times n}$ is an inverse-compound of~$M$, i.e.,~$\comp_{n-1}(A) = M$, if and only if~$A=B$ when~$n$ is even and~$A = \pm B$ when~$n$ is odd.
\end{cor}
Unfortunately, for a general $A\in\bR^{n\times m}$ and $k\leq\min(m,n)$, we are unaware of any existing closed-form solution to $A$ given $M=\comp_k(A)$. In lieu of closed-form solutions, our main result is a constructive proof for the uniqueness, up to overall sign, of the inverse-compound of~$M$ when~$\rank(M) > 1$, and an algorithm to recover the set of all inverse-compounds of~$M$.
\begin{thm}
\label{thm:main}
    Let $A\in \bR^{n\times m}$, $p=\max\{m,n\}$,  and $r := \rank(A)$. Then, the following hold:
    \begin{enumerate}
        \item if $k<r$ and $B \in \bR^{n \times m}$ fulfills~$C_k(A) = C_k(B)$, then~$A = B$ if $k$ is odd and $A = \pm B$ if $k$ is even. Moreover, there exists an algorithm whose input is~$C_k(A)$ and returns as an output $c A$ for any $c\in\{-1,1\}$ fulfilling $c^k=1$. The time complexity of the algorithm is given by:
        \begin{enumerate}
            \item $\bigO{\binom{p}{k}^3 pk + \binom{r}{k}2^r}$ if $k\leq \lceil \rank(A)/2 \rceil$
            \item $\bigO{\binom{p}{k}^3 pk + \binom{r}{k} 2^r + \binom{r}{\lceil r/2\rceil}^2 p k^2 \log r}$ if $\lceil \rank(A)/2 \rceil < k < \rank(A)$
        \end{enumerate}
        \item For $k=r$, there exists an algorithm that determines an element from the set of inverse-compounds, as defined in \cref{prop:compound_nonunique}, in time complexity $\bigO{\binom{p}{k} p^2}$.
    \end{enumerate}
\end{thm}
Note that given $\comp_k(A)$, any of its inverse-compounds has the same rank due to \cref{thm:compound_rank}. In particular, the assumption $k < \rank(A)$ is equivalent to $\rank(M) > 1$.

\begin{rem}
    For the special case of a square matrix, where $\comp_k(A) \in \bR^{\binom{n}{k} \times \binom{n}{k}}$ is invertible and $k > n/2$, there is an alternative method based on \cite{dalin2023verifying} for computing an inverse compound with better time complexity than applying \cref{thm:main} directly. Let~$U \in \{-1,0,1\}^{\binom{n}{k} \times \binom{n}{k}}$ be the matrix from \cite[Definition 3]{dalin2023verifying}, which is orthonormal, i.e., ~$U^\top U = UU^\top = I$. Then, by~\cite[Theorem 8]{dalin2023verifying},
    \begin{equation}
        \comp_k(A)^\top U^\top \comp_{n-k}(A) U = \det(A) I,
    \end{equation}
    so
    \begin{equation}
        \left(\comp_{n-k}(A)\right)^{-1} = \det(A)^{-1} U \comp_k(A)^\top U^\top.
    \end{equation}
    Finally, since~$k > n/2$, i.e., $n-k < n/2$, computing an inverse-compound for~$\left(\comp_{n-k}(A)\right)^{-1}$ using \cref{thm:main} has better time complexity, and the matrix inverse of the resulting inverse-compound is an inverse-compound of~$\comp_k(A)$ by item 1 of \cref{thm:compoundProperty}.
\end{rem}
The algorithm in \cref{thm:main} is based on a recovery method that exploits the multiplicative property of compounds and the 
\emph{reduced SVD} (also known as compact SVD): for $A\in \bR^{n\times m}$ with $\rank(A)=r$
    \begin{equation}
    \label{eq:svd-decomp}
        A = V \Sigma W^\top \implies
        M := \comp_k(A) = \comp_k(V)\comp_k(\Sigma)\comp_k(W)^\top,
    \end{equation}
    where $\Sigma\in\bR^{r\times r}$ has positive real singular values of $A$ on its diagonal, and the corresponding left and right singular vectors are the columns of the orthonormal matrices $V\in\bR^{n\times r}$ and $W\in\bR^{m\times r}$, respectively.
The main idea is to recover the singular vectors (left and right) and singular values of an inverse-compound matrix $A$ from those of the given compound-decomposable matrix~$M$, which then allows us to compose $A$ as above. 
A detailed proof of \cref{thm:main} is deferred to \Cref{sec:proof_main}, following our discussion of all the involved components.

To begin with, we explain why it is important and, without loss of generality, possible to assume that the non‑zero singular values of $M$ are distinct.
\begin{rem}\label{rem:repeated_values}
    While every wedge product of $k$ linearly independent singular vectors of~$A$ is a singular vector of~$\comp_k(A)$, if~$\comp_k(A)$ has repeating singular values, then it will have singular vectors which are \emph{not} decomposable. This is because if~$u,v$ are linearly independent singular vectors to the same singular value, every orthonormal basis of~$\spanop\{u,v\}$ consists of singular vectors, which are generally not decomposable. \cref{ex:repeated_values} illustrates this fact. 
\end{rem}
\begin{ex}\label{ex:repeated_values}
For $M = I_6$, it follows from $I_6 = \comp_2(I_4)$ that $M$ is compound-decomposable. However, if $Q \in \bR^{6 \times 6}$ is an orthonormal matrix whose first column is
\[
    q = \frac{1}{\sqrt{2}}\begin{bmatrix}
        1 & 0 & 0 & 0 & 0 & 1
    \end{bmatrix}^\top,
\]
then $QI_6Q^\top$ is an SVD of~$I_6$, where $Q$ is not compound-decomposable, i.e, there is no matrix~$S \in \bR^{4 \times 4}$ such that~$\comp_2(S) = Q$. To see this, assume by contradiction that such a matrix does exist, and let~$u,v \in \bR^{4}$ be its first two columns, i.e., 
\[
    \comp_2(\begin{bmatrix} u & v \end{bmatrix}) = q.
\]
In particular, it must hold that
\[
    \det \left(\begin{bmatrix}
        u_1 & v_1 \\
        u_3 & v_3
    \end{bmatrix}\right) = 0,
\]
i.e., $\begin{bmatrix}u_1 & u_3\end{bmatrix}^\top$ and~$\begin{bmatrix}v_1 & v_3\end{bmatrix}^\top$ are linearly dependent. The same must be true for~$\begin{bmatrix}u_2 & u_3\end{bmatrix}^\top$ and~$\begin{bmatrix}v_2 & v_3\end{bmatrix}^\top$. However, this implies that~$\begin{bmatrix}u_1 & u_2\end{bmatrix}^\top$ and~$\begin{bmatrix}v_1 & v_2\end{bmatrix}^\top$ are linearly dependent and, therefore,
\[
    \det \left(\begin{bmatrix}
        u_1 & v_1 \\
        u_2 & v_2
    \end{bmatrix}\right) = 0 \neq q_1 = \frac{1}{\sqrt{2}},
\]
which is a contradiction.
\end{ex}
To avoid getting non-decomposable singular vectors of $M=\comp_k(A)$ by reduced SVD, it is sufficient to have $M$ with distinct positive singular values. In case $M$ does not satisfy the requirement, the following \cref{prop:preprocess-distinct} provides a preprocessing routine to ensure this property. 
    
\begin{prop}
\label{prop:preprocess-distinct}
    Let~$M \in \bR^{\binom{n}{k} \times \binom{m}{k}}$ be a compound-decomposable matrix. Then, for almost all~$Q \in \bR^{n \times n}$, the non-zero singular values of~$\comp_k(Q)M$ are distinct.
\end{prop}
\begin{proof}
     The claim is trivial for $M = 0$. Therefore, let us suppose that $\comp_k(A) = M \neq 0$ with $\rank(M) \geq 1$. Since it follows from \cref{thm:compound_rank} that~$r = \rank(A) \ge k$, let a reduced SVD of $A$ be given by $A = U \Sigma V^\top$, i.e., $U^\top U = V^\top V = I$ and $\Sigma \in \mathds{R}^{r \times r}$ is diagonal and positive definite. Then, since for any $Q \in\bR^{n \times n}$ the singular values of $\comp_k(Q)M$ are the square roots of the eigenvalues of \begin{equation*}
        (\comp_k(Q) M)^\top (\comp_k(Q) M) = \comp_k(V)\comp_k(\Sigma U^\top Q^\top Q U \Sigma) \comp_k(V)^\top,
     \end{equation*}
     where $\comp_k(V)^\top \comp_k(V) = I$, our claim follows if we can show that $B :=\comp_k(\Sigma U^\top Q^\top Q U \Sigma)$ has no repeated eigenvalues for almost all $Q \in\bR^{n \times n}$. 

To this end, note that as each entry of $\Sigma U^\top Q^\top Q U \Sigma$ is a polynomial of the entries of $Q$, the same must apply to $B$. Hence, by \cite{Parlett2002MatDisc} the discriminant of the characteristic polynomial of $\comp_k(B)$ is also a polynomial of the entries of $Q$. In particular, if we can show that the this polynomial is not identical to zero, then it follows by \cite{Mityagin2020ZeroSet} or \cite[Coroally~10, p. 9]{GunningRossi}, that the set of $Q$ that make the discriminant zero, i.e., cause repeated eigenvalues, is of Lebesgue measure zero. To see this, we will show that there exists an $\alpha \in \mathds{R}^r$ such that $Q = S \Sigma^{-1} U^\top$ with
\begin{equation*}
    S = \begin{bmatrix}
    D \\
    0
\end{bmatrix} \in \mathds{R}^{n \times r} \quad \text{and} \quad D:= \diag(\exp(\alpha)) \in \mathds{R}^{r \times r},
\end{equation*}
leads to non-repeated eigenvalues of $B = \comp_k(S^\top S)$. In particular, by item 3 of~\cref{thm:compoundProperty}, $B$ has repeated eigenvalues $e^{2\sum_{i \in I} \alpha_i}$, $I \in \cI_{r,k}$, if and only if
 \begin{equation}
        \sum_{i \in I} \alpha_i = \sum_{j \in J} \alpha_j
 \end{equation}
for some $I,J \in \cI_{r,k}$ such that~$I \neq J$. Since $\bR^r$ is not the union of a finite number of subspaces, such $I$ and $J$ cannot exist for all $\alpha \in \bR^r$, which is why our claim follows.
\end{proof}

\cref{prop:preprocess-distinct} facilitates the pre-processing of $M$, i.e.,  one can first solve the inverse-compound problem for an alternative compound-decomposable $\tilde M = \comp_k(Q)M$, with an arbitrary invertible $Q\in\bR^{n\times n}$ such that $\tilde{M}$ has distinct non-zero singular values. We can then obtain $A=Q^{-1}\tilde{A}$ by finding a matrix $\tilde{A}$ with $\tilde{M}=\comp_k(\tilde{A})$, which solves our original inverse-compound problem. 
In our discussion on finding an inverse-compound, we, therefore, make the following assumption. 
\begin{ass}
\label{ass:svdsval}
    The non-zero singular values of $\comp_k(A)$ are distinct.
\end{ass}
This \cref{ass:svdsval} ensures that the left and the right singular vectors of $\comp_k(A)$ coincide with the columns of $\comp_k(V)$ and $\comp_k(W)$, respectively, up to simultaneous sign reversal, as formally stated in the following.
\begin{lem}
\label{prop:distinct-sval-svd}
    Let $A\in\bR^{n\times m}$ have $\rank(A)=r$, and $A=V\Sigma W^\top$ be a reduced SVD. Further, let $k\leq r$, such that $M=\comp_k(A)=LSR^\top$ is a reduced SVD with $S$ having distinct non-zero diagonal entries. Then $L=\comp_k(V)P\tilde{S}$ and $R=\comp_k(W)P\tilde{S}$ for a permutation matrix $P\in\{0,1\}^{\binom{r}{k}\times \binom{r}{k}}$ and diagonal sign matrix $\tilde{S}$ such that $\tilde{s}_{ii}\in\{1,-1\}$ for all $i$.
\end{lem}
The following outlines the main steps of our algorithm for finding an inverse-compound of $M=\comp_k(A)$ with some $A\in\bR^{n\times m}$, $k<\rank(A)$, where $M=LSR^\top$ and $A=V\Sigma W^\top$ are the reduced SVDs: 
\begin{enumerate}
    \item \textbf{Wedge-decomposition:} Given a set of non-zero $k$-wedges $\{z_i\}_{i=1}^{\binom{n}{k}}$ such that $z_i = u_{i_1}\wedge\cdots\wedge u_{i_k}$, where $(i_1,\ldots,i_k)\in\cI_{n,k}$ is the $i$th element, find $\{\pm u_1,\ldots,\pm u_n\}$ with $\|u_{j\in[n]}\|_2=1$.
    Leveraging this fundamental step, we recover $\hat{V}=\{\pm V_{:,i}\}_{i\in[r]}$ and $\hat{W}=\{\pm W_{:,i}\}_{i\in[r]}$ from $L$ and $R$ respectively. Note that $\hat{V}$ and $\hat{W}$ are not guaranteed to have sign-consistent columns.
    \item \textbf{Ordering singular values:} Since $S$ is not guaranteed to be compound-decomposable, use $\comp_k(\hat{V})$ and $M$ to recover $\comp_k(\hat\Sigma)$, such that $\diag(\hat\Sigma)$ is a permutation of $\diag(\Sigma)$. The diagonal of $\comp_k(\hat\Sigma)$ is a permutation of $\diag(S)$.
    \item \textbf{Sign-consistency restoration:} Adjust the column signs of $\hat{W}$ using $\comp_k(\hat{\Sigma})$, $\comp_k(\hat{V})$, $L$ and $R$, giving $\tilde{W}$ as output. $\tilde{W}$ is consistent in column signs with $\tilde{V} := \hat{V}$.
    \item \textbf{Singular values recovery:} Finally, from the diagonal of $\comp_k(\hat\Sigma)$ recovered in step $2$, find $\hat\Sigma$ by solving a linear system of equations.
\end{enumerate}
Composing the components as determined above, it follows that $\tilde{V}\hat\Sigma \tilde W^\top$ is an inverse-compound of $M$.
Before we detail each of these steps, we introduce the following example of a compound-decomposable matrix $M = \comp_k(A)$, which in the remainder of the paper will be used to illustrate the above steps.
\begin{ex}
\label{ex:examplematrix}
    Let $A\in\bR^{4\times 4}$ be given by
    \begin{equation*}
        A = 
        \begin{pmatrix*}[r]
        3 & -1 & -6 & -4 \\
        3 & -3 & -2 & 4 \\
        4 & 3 & 7 & 1 \\
        -5 & -1 & -1 & 1
        \end{pmatrix*}
    \end{equation*}
    {\normalsize with $\rank(A)=3$. $A$ has a reduced SVD decomposition $A = V\Sigma W^\top$ as},
    \begin{equation*}
        A = 
        \begin{pmatrix*}[r]
        0.58 & -0.60 & 0.34 \\
        0.13 & -0.32 & -0.93 \\
        -0.79 & -0.35 & 0.08 \\
        0.18 & 0.64 & -0.10
        \end{pmatrix*}
        \begin{pmatrix*}[r]
        10.50 & 0.00 & 0.00 \\
        0.00 & 7.60 & 0.00\\
        0.00 & 0.00 & 5.92
        \end{pmatrix*}
        \begin{pmatrix*}[r]
        -0.19 & -0.97 & -0.16 \\
        -0.33 & -0.02 & 0.47 \\
        -0.90 & 0.16 & 0.08 \\
        -0.23 & 0.19 & -0.86
        \end{pmatrix*}^\top
    \end{equation*}
    {\normalsize However, $A$ is apriori unknown to us and instead we are given}
    \begin{equation*}
        M := \comp_2(A) = 
        \begin{pmatrix*}[r]
        -6 & 12 & 24 & -16 & -16 & -32 \\
        13 & 45 & 19 & 11 & 11 & 22 \\
        -8 & -33 & -17 & -5 & -5 & -10 \\
        21 & 29 & -13 & -15 & -15 & -30 \\
        -18 & -13 & 23 & 1 & 1 & 2 \\
        11 & 31 & 9 & 4 & 4 & 8
        \end{pmatrix*}.
    \end{equation*}
{\normalsize By \cref{thm:compound_rank}, $\rank(M)=3$ with reduced SVD $M = L S R^\top$, where}
\[
M = 
\begin{pmatrix*}[r]
0.11 & 0.58 & 0.67 \\
0.67 & -0.31 & 0.07\\
-0.48 & 0.12 & -0.16\\
0.29 & 0.72 & -0.35\\
-0.14 & -0.16 & 0.63\\
0.44 & -0.06 & -0.02
\end{pmatrix*} 
\begin{pmatrix*}[r]
79.80 & 0.00 & 0.00\\
0.00 & 62.12 & 0.00\\
0.00 & 0.00 & 45.01
\end{pmatrix*}
\begin{pmatrix*}[r]
0.32 & 0.14 & -0.46 \\
0.90 & 0.16 & -0.05 \\
0.26 & -0.12 & 0.87\\
0.07 & -0.40 & -0.08\\
0.07 & -0.40 & -0.08 \\
0.13 & -0.79 & -0.15
\end{pmatrix*}^\top
\]
Given $k=2$, our goal is to recover $A$ by using $L$, $S$ and $R$.
\end{ex}

\input{esvector}
\input{esvalue}

%% file: esvector.tex
\subsection{Recovering Singular Vectors from Wedge Decompositions}
\label{sec:kwedgedecomp}
Let $M\in \bR^{\binom{n}{k}\times \binom{m}{k}}$ be compound-decomposable such that $M=LSR^\top$ is a \textit{reduced SVD}, and $M$ follows \cref{ass:svdsval}. As discussed in \cref{prop:distinct-sval-svd}, $L$ and $R$ may not be compound-decomposable themselves, but only up to a column permutation and simultaneous sign reversal of a subset of their columns. This section shows that it is enough to decompose the columns of $L$ (resp. $R$) to recover the left and right singular vectors of $A$ such that $\comp_k(A)=M$. 

To this end, let $U = [u_1,\ldots,u_r]\in\bR^{n\times r}$ with $1<r\leq n$ be such that each column has unit norm. For a given $k<r$, $\comp_k(U) = [z_1,\ldots,z_{\binom{r}{k}}]$, such that the column $z_i = u_{i_1}\wedge\cdots\wedge u_{i_k}$ where the index tuple $(i_1,\ldots,i_k)$ is the $i$th element of $\cI_{r,k}$. 
The key to our algorithm is to find a decomposition of each of the $k$-wedges $\{\pm z_i\}_{i\in[\binom{r}{k}]}$ as a $k$-dimensional subspace and recover the columns of $U$, up to the sign, by intersecting the subspaces.
The following theorem gives a sufficient condition for the recovery.
\begin{thm}
\label{thm:unique-decomposition}
    Let $U=[u_1,\cdots, u_r]$ be such that any $k<r$ columns of $U$ are linearly independent, and $\|u_\ell\|_2=1$ for all $\ell\in [r]$. Then \cref{alg:ESvecRecovery}, having as input a set of $k$-wedges $\{\pm z_i\cond z_i = u_{i_1}\wedge\cdots\wedge u_{i_k}, 1\leq i\leq \binom{r}{k}\}$,
    \begin{enumerate}
        \item for each $z_i$ computes the subspace $\spanop\{u_{i_1},\ldots,u_{i_k}\}$. The total time complexity of computing all the subspaces is $\bigO{n^2\binom{n}{k+1}\binom{r}{k}}$.
        \item computes all one-dimensional pairwise intersections among the subspaces generated in step $1$, returning $\hat{u}_\ell=\pm u_\ell$ for all $1\leq \ell\leq r$ as output. The total time to compute $\hat{U}=[\hat{u}_1,\cdots,\hat{u}_r]$ is of complexity
            \begin{enumerate}
                \item $\bigO{\binom{r}{k}^2 n k^2}$ if $k\leq \lceil r/2 \rceil$.
                \item $\bigO{\binom{r}{\lceil r/2\rceil}^2 n k^2 \log r}$ if $\lceil r/2 \rceil <k < r$.
            \end{enumerate}
    \end{enumerate}
\end{thm}
The two steps of the above theorem are elaborated in the next two subsections.

\subsubsection{Creating Wedge Matrices}
\label{sec:wedgematrix}
This section provides a constructive proof of step $1$ of \cref{thm:unique-decomposition}.
Recall that for every non-zero and decomposable vector~$z$ in the $k$th exterior space~$V^{(k)}$, there exists a unique $k$-dimensional subspace~$\mathcal{U}$ in~$V$ such that if~$z = u_1 \wedge \cdots \wedge u_k$, then~$\spanop\{u_1,\dots,u_k\} = \mathcal{U}$. The following lemma provides a method that uses the entries of the vector~$z$ to construct a matrix whose kernel is exactly~$\mathcal{U}$. Given an element~$s$ in an ordered set~$S$, we will use~$\indexof(s,S)$ to denote the index of~$s$ in~$S$. For example
\begin{equation*}
    \indexof(3, (1,3,4)) = 2, \quad \text{and} \quad\indexof((1,3,4), \cI_{4,3}) = 3.
\end{equation*}
\begin{lem}
\label{thm:wedgeMat}
    Let $k < n$, $z \in \bR^{\binom{n}{k}}$ be non-zero and decomposable, and $\cU \subset \bR^n$ denote the unique $k$-dimensional subspace in~$\bR^n$ such that if~$z = u_1 \wedge \cdots \wedge u_k$ then~$\spanop\{u_1,\dots,u_k\} = \cU$. Define the matrix~$M_z \in \bR^{\binom{n}{k+1} \times n}$ by
    \begin{equation}\label{eq:Mz_def}
        (M_z)_{i,j} = \begin{cases}
            0, & j \notin I, \\
            (-1)^{\indexof(j,I)-1} z_{\indexof(I \setminus \{j\}, \cI_{n,k})}, & \text{otherwise},
        \end{cases}
    \end{equation}
    where~$I$ is the $i$th element of~$\cI_{n,k+1}$. Then,~$\ker(M_z) = \mathcal{U}$.
\end{lem}
\begin{rem}
    To compute~$\indexof(I, \cI_{n,k})$, one can either perform a simple linear search on the ordered set~$\cI_{n,k}$, or use the explicit formula in the appendix of~\cite{Ofir2023KroneckerCompounds}.
\end{rem}
\begin{proof}
    Recall that~$x \in \mathcal{U}$ if and only if~$x \wedge z = 0$. Since the exterior product is multilinear, the mapping~$f : \bR^n \to \bR^{\binom{n}{k+1}}$ defined by $f(x) := x \wedge z$ is linear, i.e.,  there exists a matrix~$M_z \in \bR^{\binom{n}{k+1} \times n}$ such that~$f(x) = M_z x$, and~$\ker(M_z) = \mathcal{U}$. We will prove that the matrix defined in~\eqref{eq:Mz_def} is a matrix representation of~$f(x)$ for the canonical basis, i.e., the~$j$ column of~$M_z$ is equal to~$f(e_j)$.

    To this end, let $u_1,\dots,u_k \in \bR^n$ be such that~$z = u_1 \wedge \cdots \wedge u_k$. Fix~$i \in \{1,\dots,\binom{n}{k+1}\}, \;j \in \{1,\dots,n\}$. 
    In Grassmann notation, the~$i$th entry of~$f(e_j)$ is given by
    \begin{equation*}
        (f(e_j))_i = \det (W_{I,1:k+1})
    \end{equation*}
    where~$W = \begin{bmatrix} e_j & u_1 & \dots & u_k \end{bmatrix}$ and~$I=(i_1,\dots,i_{k+1})$ is the~$i$th element of~$\cI_{n,k+1}$. If~$j \notin I$, all entries of the first column of the submatrix~$W_{I,1:k+1}$ are zeros, so~$(f(e_j))_i = 0$. Suppose now that~$j \in I$ and $\ell$ is such that~$I_\ell = j$, or equivalently $\ell = \indexof(\ell, I)$. Then~$W_{I,1:k+1}$ takes the form
    \begin{equation*}
        W_{I,1:k+1} = 
        \begin{bmatrix*}
            0 & u_{i_1 1} & \dots & u_{i_1 k} \\
            \vdots & & & \vdots \\
            0 & u_{i_{\ell-1} 1} & \dots & u_{i_{\ell-1} k} \\
            1 & u_{i_\ell 1}  & \dots & u_{i_\ell k} \\
            0 & u_{i_{\ell+1} 1} & \dots & u_{i_{\ell+1} k} \\
            \vdots & & & \vdots
        \end{bmatrix*}
    \end{equation*}
    By the properties of determinants,
    \begin{align*}
        \det (W_{I,1:k+1})& = (-1)^{\ell-1} \det \left(
        \begin{bmatrix}
            1 & u_{i_\ell 1}  & \dots & u_{i_\ell k} \\
            0 & u_{i_1 1} & \dots & u_{i_1 k} \\
            \vdots & & & \vdots \\
            0 & u_{i_{\ell-1} 1} & \dots & u_{i_{\ell-1} k} \\
            0 & u_{i_{\ell+1} 1} & \dots & u_{i_{\ell+1} k} \\
            \vdots & & & \vdots
        \end{bmatrix}\right) \\
        &= (-1)^{\ell-1} \det \left(
        \begin{bmatrix}
            u_{i_1 1} & \dots & u_{i_1 k} \\
            \vdots & & \vdots \\
            u_{i_{\ell-1} 1} & \dots & u_{i_{\ell-1} k} \\
            u_{i_{\ell+1} 1} & \dots & u_{i_{\ell+1} k} \\
            \vdots & & \vdots
        \end{bmatrix} \right)\\
        &= (-1)^{\ell-1} z_{\indexof(\{i_1, \dots, i_{\ell-1}, i_{\ell+1}, \dots, i_k\}, \cI_{n,k})}.
    \end{align*}
    Since~$I_\ell = j$ by definition, i.e., $I=(i_1, \dots, i_{\ell-1}, j, i_{\ell+1}, \dots, i_k)$, it follows that~$(f(e_j))_i = (M_z)_{i,j}$ as claimed.
\end{proof}
Using the construction presented in the above \cref{thm:wedgeMat}, we can construct a set of wedge matrices corresponding to the columns of $\comp_k(U)$, thus proving step $1$ of \cref{thm:unique-decomposition}.
\begin{rem}\label{rem:ker_M_z}
    The time complexity for the wedge matrix construction given by \cref{thm:wedgeMat} for a decomposable vector $z$ in the exterior space of dimension $\binom{n}{k}$ is linear in the size of $M_z\in \bR^{\binom{n}{k+1}\times n}$. Moreover, computing the $k$-dimensional subspace corresponding to $M_z$ involves computing $\ker(M_z)$. A basis of unit vectors of $\ker(M_z)$ can be computed via its SVD $M_z=U_1\Lambda U_2^T$, given by the columns of $U_2$ that correspond to zero singular values, which amounts to a time complexity of $\bigO{\binom{n}{k+1}n^2}$ (see \cite{trefethen1997numerical}). Thus, for all the columns of $\comp_k(U)$, the total time complexity becomes $\bigO{\binom{n}{k+1}\binom{r}{k} n^2}$.
\end{rem}
The following example demonstrates the creation of a wedge matrix described above.
\begin{ex}
    Let us consider $U\in\bR^{4\times 4}$ as,
    \[
    U = 
    \begin{bmatrix}
        u_1 & u_2 & u_3 & u_4
    \end{bmatrix}
    \]
    where $u_i=[u_{1i}\ u_{2i}\ u_{3i}\ u_{4i}]^\top$ for $1\leq i\leq 4$.
    Now for $k=2$, consider the compound
    \[\comp_2(U) =
    \begin{bmatrix}
        z_1 & z_2 &\ldots & z_6
    \end{bmatrix} \in \bR^{6\times 6}.
    \]
    The column $z_1$ can be expressed as
    \begin{align*}
    z_1 
    &= u_1\wedge u_2\\
    &= (u_{11}e_1 + u_{21}e_2 + u_{31}e_3 + u_{41}e_4) \wedge (u_{12}e_1 + u_{22}e_2 + u_{32}e_3 + u_{42}e_4)\\
    &= (u_{11}u_{22}-u_{21}u_{12})(e_1\wedge e_2) + (u_{11}u_{32}-u_{31}u_{12}) (e_1\wedge e_3) + \\
    &\phantom{,,,} (u_{11}u_{42}-u_{41}u_{12})(e_1\wedge e_4) +
    (u_{21}u_{32}-u_{31}u_{22})(e_2\wedge e_3) + \\
    &\phantom{,,,} (u_{21}u_{42}-u_{41}u_{22})(e_2\wedge e_4) + (u_{31}u_{42}-u_{41}u_{32})(e_3\wedge e_4)\\
    &= \sum_{i=1}^6 z_{i 1} \cdot e_{I_i},
    \end{align*}
    where $I_i$ denotes the $i$th element of $\cI_{4,2}$, and $e_{I_i}$ denotes the corresponding basis of the exterior product space ${\mathds{R}^{4}}^{(2)}$, e.g., $e_{I_1}= e_1\wedge e_2$.
    Since $u_1 \wedge z_1 = 0$ (similarly for $u_2$) and
    \begin{align}
    \label{eq:extprodzero}
    u_1 \wedge z_1
    &= (u_{11}e_1 + u_{21}e_2 + u_{31}e_3 + u_{41}e_4) \wedge z_1 \notag\\
    &= (u_{11}z_{4 1}-u_{21}z_{2 1}+u_{31}z_{1 1})(e_1\wedge e_2\wedge e_3) + \notag \\
    & \phantom{,,,} (u_{11}z_{5 1}-u_{21}z_{3 1}+u_{41}z_{1 1})(e_1\wedge e_2\wedge e_4) + \notag\\
    &\phantom{,,,} (u_{11}z_{6 1}-u_{31}z_{3 1}+u_{41}z_{2 1})(e_1\wedge e_3\wedge e_4) + \notag\\
    &\phantom{,,,} (u_{21}z_{6 1}-u_{31}z_{5 1}+u_{41}z_{4 1})(e_2\wedge e_3\wedge e_4)\notag\\
    &= 0,
    \end{align}
    it must hold that every coefficient equals zero.
    Thus, defining our wedge matrix $M_{z_1}\in \bR^{4\times 4}$ corresponding to $z_1$ as
    \[
    M_{z_1} = 
    \begin{pmatrix*}[r]
        z_{4 1} & -z_{2 1} & z_{1 1} & 0\\
        z_{5 1} & -z_{3 1} & 0 & z_{1 1}\\
        z_{6 1} & 0 & -z_{3 1} & z_{2 1}\\
        0 & z_{6 1} & -z_{5 1} & z_{4 1}
    \end{pmatrix*},
    \]
    i.e., the condition for the entries of $u_1=[u_{11},u_{21},u_{31},u_{41}]^\top$ in \cref{eq:extprodzero} is equivalent to $u_1\in \ker(M_{z_1})$, and similarly, $u_2\in \ker(M_{z_1})$.
\end{ex}
Thus, from $\comp_k(U)$ for $U\in\bR^{n\times r}$, we are able to recover all the $k$-dimensional subspaces spanned by every $k\leq r$ columns of $U$. Notice that $k=r$ gives just one subspace, which is $\im(U)$.
Our next goal is to extract $\{u_1,\ldots,u_r\}$ from the $k$-dimensional subspaces.

\subsubsection{Recovery by Wedge Intersections}
\label{sec:evrecovery}

\begin{algorithm}
    \small
    \caption {Wedge Decomposition and Intersection for $U\in\bR^{n\times r}$ and $k<r$}
    \label{alg:ESvecRecovery}
    \begin{algorithmic}[1]
    \Procedure{Main}{}
        \State \textbf{Input}: $\{z_1, \ldots,  z_{\binom{r}{k}}\}$ s.t. $\comp_k(U)=[\hat{z}_1, \cdots,  \hat{z}_{\binom{r}{k}}]$ where $\forall i\in[\binom{r}{k}]$, $z_i\in\{\hat{z}_i,-\hat{z}_i\}$, $n, r, k$.
        \State \textbf{Output}: Columns of $\hat{U}=[\hat{u}_1,\cdots, \hat{u}_r]$ where $\hat{u}_i=\pm u_i$ for all $i\in[r]$.
        \State \textbf{Initialize} $\hat{U}$ as a null matrix.
        \State Compute wedge matrices $M_{z_i}\in\bR^{\binom{n}{k+1} \times n}$ for $1\leq i\leq\binom{r}{k}$, as given in \cref{thm:wedgeMat}.
        \For{$i\gets 1,\ldots,\binom{r}{k} $}
            \State Set $S_i=\ker(M_{z_i})$ s.t. $S_i = \im(B_i)$ with basis $B_i$;
        \EndFor
        \State Set $\cS^0 \gets [\operatorname{vec}(B_1),\ldots,\operatorname{vec}(B_{\binom{r}{k}})]$
        \If{$k\leq \lceil r/2 \rceil$}
        \For{$i \gets 1,\ldots,|\cS^0|$}
            \For{$j \gets i+1,\ldots,|\cS^0|$}
            \State $B \gets \Call{Intersect}{S_i, S_j, n}$; \Comment{$S_i$: $i$th element of $\cS^0$}
            \If{$B$ has single column $\&$ $B_{:,1}\notin \hat{U}$ $\&$ $-B_{:,1}\notin \hat{U}$}
            \State $\hat{U}\gets [\hat{U}, B_{:,1}]$; \Comment{retain 1-D intersections}
            \EndIf
            \EndFor
        \EndFor
        \Else
        \State Set $s = 2k-r$; \Comment{min-intersection dimension}
        \While{$s > \lceil r/2 \rceil$}
            \State $\cS_{next}\gets\emptyset$; $\cS\gets\cS^0$;
            \For{$i \gets 1,\ldots,|\cS|$}
            \For{$j \gets i+1,\ldots,|\cS|$}
                \State $B \gets \Call{Intersect}{S_i, S_j, n}$; \Comment{$S_i:$ ith element of $\cS$}
                \If{$B$ has $s$ columns}
                    \State $\cS_{next}\gets [\cS_{next}, \operatorname{vec}(B)]$; \Comment{update current intersection set}
                \EndIf
            \EndFor
            \EndFor
            \State $\cS \gets \cS_{next}$; $s = 2s-r$; \Comment{update dim(min-intersection)}
        \EndWhile
        \For{$i \gets 1,\ldots,|\cS|$}
            \For{$j \gets i+1,\ldots,|\cS|$}
                \State $B \gets \Call{Intersect}{S_i, S_j, n}$; 
                \If{$B$ has single column $\&$ $B_{:,1}\notin \hat{U}$ $\&$ $-B_{:,1}\notin \hat{U}$}
                \State $\hat{U}\gets [\hat{U}, B_{:,1}]$;
                \EndIf
            \EndFor
        \EndFor
        \EndIf
        \State \textbf \Return $\hat{U}$
    \EndProcedure

    \Function{Intersect}{$S_1$, $S_2$, $n$}
    \State Set $B_1\gets reshape(S_1,n)$ and $B_2\gets reshape(S_2,n)$; \Comment{$n$-dim basis}
    \State Set $k\gets$ number of columns in $B_1$;
    \State $N \gets \ker([B_1 \ -B_2])$;
    \State $B \gets B_1 N_{1:k,:}$; \Comment{intersected subspace basis}
    \State \Return $B$;
    \EndFunction
    \end{algorithmic}
\end{algorithm}

This section provides a constructive proof of step $2$ of \cref{thm:unique-decomposition}.
Now that we have determined the $k$-dimensional subspaces, given by $\spanop\{u_{i_1},\ldots,u_{i_k}\}$ for all $1\leq i_1<\cdots<i_k\leq r$, we next discuss a method which finds $\{\pm u_1,\ldots,\pm u_r\}$ given by the basis of the subspace intersections, which implicitly requires $k<r$.
For $k=r$, the previous subsection generates only one subspace $\cU = \spanop\{u_1,\ldots,u_r\}$ corresponding to the single column of $\comp_k(U)$. Any basis of $\cU$ will give a member from the set $\{V\in\bR^{n\times r} \cond \comp_k(V)=\comp_k(U)\}$, due to the decomposability property in \Cref{sec:extAlgebra}, i.e., $\hat{U} = U R$ where $R\in \bR^{r\times r}$ is an invertible change-of-basis matrix in $\cU$.

The subsequent discussion assumes $k<r$. For the sake of completeness, we present now the closed-form solution to the intersection of two subspaces.
\vspace{2mm}

\textbf{Basis construction of two Subspace Intersection:}
Consider two $k$-dimensional subspaces $S_1, S_2 \subseteq  \bR^n$ with basis given by $B_1,B_2\in \bR^{n\times k}$, respectively, such that, $S_1 = \im(B_1)$ and $S_2 = \im(B_2)$. Then, $x \in S := S_1 \cap S_2$ if and only if there exist $\alpha_1,\alpha_2\in \bR^k$ fulfilling
\begin{equation}
\label{eq:2subspaceint}
\begin{bmatrix}
    B_1 & -B_2
\end{bmatrix}
\begin{bmatrix}
    \alpha_1\\
    \alpha_2
\end{bmatrix}
=0 \quad \Longleftrightarrow  \quad \begin{bmatrix}
    \alpha_1\\
    \alpha_2
\end{bmatrix} \in \ker([B_1\ -B_2 ]).
\end{equation}
In particular, if the columns of $N\in \bR^{2k\times s}$ represent a basis of $\ker([B_1\ -B_2])$, then the columns of $B := \begin{bmatrix}
    B_1 N_{1:k,1} & \cdots & B_1 N_{1:k,s}
\end{bmatrix}$ are a basis of $S$. 
\begin{rem} \label{rem:complexity_null}
  Computing the basis of $S_1\cap S_2$, where $S_1,S_2\subseteq \bR^n$ are of dimension $k$, involves the SVD of a $(n\times 2k)$-dimensional matrix as described above, which analogous to \cref{rem:ker_M_z} amounts to a time complexity of $\bigO{nk^2}$.
\end{rem}
We proceed by discussing the pairwise-intersection method among the subspaces contained in $\cS^0=\{\spanop\{u_{i_1},\ldots,u_{i_k}\}\cond 1\leq i_1<\cdots<i_k\leq r\}$, generated in the previous subsection, each corresponding to a column of $\comp_k(U)$. This method, elaborated in \cref{alg:ESvecRecovery}, recovers $\hat{U}$ as stated in the second step of \cref{thm:unique-decomposition} by computing one-dimensional intersections among the subspaces in $\cS^0$.

\begin{proof}[A constructive proof of \cref{thm:unique-decomposition} (step $2$):] 
Given the columns of $\comp_k(U)\in\bR^{\binom{n}{k}\times \binom{r}{k}}$ as input, depending on the value of $k$ and $r$, \cref{alg:ESvecRecovery} works on the following two cases:
\begin{itemize}
    \item \textbf{Case $1$} ($k\leq \lceil r/2 \rceil$): For any $u_i$, $i\in[r]$, there exist two $k$-dimensional subspaces $S_1, S_2\in\cS^0$, such that $S_1\cap S_2 = \spanop\{u_i\}$. As $\|u_i\|_2=1$ by assumption, such intersection recovers $\hat{u_i} = \pm u_i$. However, as we are unaware of a correct lexicographical order of the columns of $\comp_k(U)$, we have to exhaustively examine all $\binom{r}{k}^2$ intersections among the elements of $\cS^0$, which as discussed in \cref{rem:complexity_null}, comes with a total time complexity of $\bigO{\binom{r}{k}^2 nk^2}$.
    \item \textbf{Case $2$} ($\lceil r/2 \rceil<k<r$): We run the following iterative algorithm. Set $s_0=k$.
    \begin{enumerate}
    \item In round $t=1, 2, \ldots$, set $s_t:= \min_{i\neq j}\dim(S_i\cap S_j) = 2s_{t-1}-r$, where $S_i,S_j\in\cS^{t-1}$.
    Construct $\cS^t:=\{S_i\cap S_j\cond S_i,S_j\in\cS^{t-1}, \dim(S_j\cap S_j)=s_t\}$, where $|\cS^t|=\binom{r}{s_t}$ and $s_t = 2^tk-(2^t-1)r$. With $k\leq r-1$, this gives $s_t\leq r - 2^t$, i.e., $s_t$ decreases exponentially with iteration $t$, upper-bounding the maximum number of iterations to $t_{max}=\cO(\log r)$.
    Further, for any $t$, $|\cS^t|=\binom{r}{s_t}\leq \binom{r}{\lceil r/2\rceil}$.
    \item Run step $1$ until $s_t\leq \lceil r/2\rceil$, at which time we can recover $\hat{U}$ as in Case $1$ above. Hence, the total time complexity of Case $2$ is $\bigO{\binom{r}{\lceil r/2\rceil}^2 nk^2\log r}$.
    \end{enumerate}
    \end{itemize}
Thus, for a $k<r$, the above method (deterministically) computes $\hat{u}_i=\pm u_i$ for all $i\in[r]$ as unit norm bases of the one-dimensional intersections, which are unique up to the sign, for a particular $\comp_k(U)$.
\end{proof}

Now, with the general result given by \cref{thm:unique-decomposition}, \cref{alg:ESvecRecovery} can be used to recover the singular vector matrices from their respective compound.
Let $M$ be a compound-decomposable matrix such that $M=\comp_k(A)$ for an $A\in\bR^{n\times m}$, and $M=LSR^\top$ is a compact SVD. To recover the left and right singular vectors of $A$, \cref{alg:ESvecRecovery} is run twice, independently with the columns of $L$ and $R$ as respective input. Note that we maintain the same order over the columns of $L$ and $R$ in the input, so that the left and right singular vectors of $A$, as generated by the algorithm, are order-consistent. The recovery conditions are formally stated in the following.
\begin{prop}
\label{prop:sign-inconsistency-svd}
     Let $M=\comp_k(A)$ for an $A\in\bR^{n\times m}$ of rank $r>k$ such that $M$ has distinct non-zero singular values. Suppose $A=V\Sigma W^\top$ and $M=L S R^\top$, given by a reduced SVD. Then, with a fixed-ordered sequence of the columns of $L$ and $R$, \cref{alg:ESvecRecovery} respectively outputs $\hat{V}$ and $\hat{W}$ such that $\hat{V}=VPS_1$ and $\hat{W}=WPS_2$, where $P\in \{0,1\}^{r\times r}$ is a permutation matrix and $S_1,S_2\in\bR^{r \times r}$ are diagonal with $\diag(S_1),\diag(S_2)\in\{1,-1\}^r$.
\end{prop}
With the distinct non-zero singular value \cref{ass:svdsval} on $M$, the columns of $L$ and $R$ coincide with the columns of $\comp_k(\hat{V})$ and $\comp_k(\hat{W})$ respectively, up to the sign. Because of the orthonormality of singular vectors, the unit-norm and independence conditions in \cref{thm:unique-decomposition} are met.
For \cref{ex:examplematrix} we get
\begin{equation}
\label{eq:VrWr}
\hat{V} = 
\begin{pmatrix*}[r]
0.58 & -0.60 & -0.34 \\
0.13 & -0.32 & 0.93 \\
-0.79 & -0.35 & -0.08 \\
0.18 & 0.64 & 0.10
\end{pmatrix*}
\quad
\hat{W} = 
\begin{pmatrix*}[r]
0.19 & 0.97 & 0.16 \\
0.33 & 0.02 & -0.47 \\
0.90 & -0.16 & -0.08 \\
0.23 & -0.19 & 0.86
\end{pmatrix*}
\end{equation}
However, to compose $A$ using the recovered singular vectors, i.e., the columns of $\hat{V}$ and $\hat{W}$, generated in separate runs of \cref{alg:ESvecRecovery}, they must be order and sign consistent.
The order consistency guarantee comes from a fixed permutation matrix $P$ in \cref{prop:sign-inconsistency-svd}. Hence, the remaining task is to restore sign consistency, i.e., to find a diagonal sign-adjustment matrix $T$ with $\diag(T)\in \{1,-1\}^r$ such that $S_1=S_2 T$, which is discussed in \Cref{sec:sign_inconsistency}.

\vspace{2mm}

\subsection{Ordering the Compound Singular Values}
\label{sec:restoreMap}
Since for a reduced SVD $M=\comp_k(A)=LSR^\top$ of $M$, $S$ is not necessarily compound-decomposable,  we are required to find a compound-decomposable permutation of $\diag(S)$, where $A=V\Sigma W^\top$ is a reduced SVD of $A$. Since we have already recovered $\hat{V}$, as in \cref{prop:sign-inconsistency-svd}, we can obtain such a permutation $\comp_k(\hat\Sigma)$ satisfying $AA^\top \hat{V} = \hat{V}\hat\Sigma^2$ from 
    \begin{equation}
    \label{eq:hatSigma}
            \comp_k(\hat\Sigma)^2 = \comp_k(\hat\Sigma^2)=
            \comp_k((AA^\top\hat{V})^\top \hat{V}) = 
            \br{MM^\top\comp_k(\hat{V})}^\top \comp_k(\hat{V}).
    \end{equation}
Note that $\hat{\Sigma}$ is such that $\diag(\hat{\Sigma})$ is a permutation of $\diag(\Sigma)$ that is consistent with $\hat{V}$.
Thus, the reordering of $\diag(S)$ is given by the diagonal of $\comp_k(\hat\Sigma)$. 
For \cref{ex:examplematrix}, using the recovered $\hat{V}$ in \cref{eq:VrWr}, we obtain
\begin{equation*}
\label{eq:sigmar}
    \comp_2(\hat\Sigma) = 
    \begin{pmatrix*}[r]
    79.80 & 0.00 & 0.00 \\
    0.00 & 62.12 & 0.00 \\
    0.00 & 0.00 & 45.01
    \end{pmatrix*}.
\end{equation*}
As shown in \Cref{sec:ESvalueRecover}, $\comp_k(\hat{\Sigma})$ can then be used to recover the singular values of $A$.
\begin{rem}
\label{rem:time-OrderCompSval}
The time complexity for reordering the singular values of $M=\comp_k(A)$, as described in this subsection, involves two components as follows: 
\begin{enumerate}
    \item computation of $\comp_k(\hat{V})$ for $\hat{V}\in\bR^{n\times r}$, which amounts to $\bigO{k^3\binom{n}{k}\binom{r}{k}}$, due to the computation of $\det(\hat{V}_{I,J})$ for all $I\in \cI_{n,k}, J\in \cI_{r,k}$,
    \item matrix multiplication as in \cref{eq:hatSigma}, which is of time complexity $\bigO{\binom{n}{k}^2\binom{m}{k} + \binom{n}{k}^2\binom{r}{k}}$,
\end{enumerate}
thus amounting to the overall time complexity of $\bigO{\binom{p}{k}^3 k^2}$ where $p=\max\{m,n\}$.
\end{rem}
\subsection{Recovering Sign Consistency of Singular Vectors}
\label{sec:sign_inconsistency}
In \cref{prop:sign-inconsistency-svd}, we have seen the possible sign inconsistency between the recovered left and right singular vectors, $\hat{V}$ and $\hat{W}$ respectively, for $A$ with compact SVD $A=V\Sigma W^\top$. For $M=\comp_k(A)$ with compact SVD $M=LSR^\top$, $L$ and $R$ coincide with $\comp_k(V)$ and $\comp_k(W)$, respectively, up to identical column permutation and simultaneous sign inversions of the columns, i.e., $L=\comp_k(V) P S_0$ and $R=\comp_k(W) P S_0$ for some permutation matrix $P$ and diagonal sign matrix $S_0$. Using the distinct compound singular values given by $\comp_k(\hat{\Sigma})$, as obtained in \cref{eq:hatSigma}, we align the columns of $L$, $R$ with $\comp_k(\hat{V})$, $\comp_k(\hat{W})$ respectively, adhering to a lexicographical order. Henceforth, we will refer to this order while indexing the columns of $L$ or $R$.

Fortunately, using the sign-coupling between $L$ and $R$, we can restore sign consistency between the left and the right singular vectors of $A$. The following \cref{thm:signconsistency} states the existence of such a sign-adjustment algorithm.
\begin{prop}
\label{thm:signconsistency}
    Let $M=\comp_k(A)$ for an $A\in\bR^{n\times m}$, $p=\max\{m,n\}$,  of rank $r>k$ with compact SVDs $A=V\Sigma W^\top$, $M=L S R^\top$. 
    Let $\tilde{V} := \hat{V}=V S_1$, $\hat{W}=W S_2$ for some diagonal sign matrix $S_1,S_2\in\bR^{r\times r}$ such that $\diag(S_1),\diag(S_2)\in\{1,-1\}^r$.
    With $L,R, \tilde{V}, \hat{W}$ as input, \cref{alg:signadjustsvd} finds a diagonal sign matrix $S_3$ with $\diag(S_3)\in \{1,-1\}^r$, such that $\tilde{W}=\hat{W}S_3$ satisfies $A = \tilde{V} \Sigma \tilde{W}^ \top$.
\end{prop}
\vspace{2mm}
\noindent The working principle of \cref{alg:signadjustsvd} has the following components:
\begin{enumerate}
    \item Compute $\comp_k(\tilde{V})$.
    \item Find the set $J = \left\{i\cond L_{:,i}=-[\comp_k(\hat{V})]_{:,i} , 1\leq i\leq \binom{r}{k} \right\}$.
    \item For $j\in J$, $R_{:,j}=-R_{:,j}$.
    \item Find diagonal $S_3$ with $\diag(S_3)\in\{1,-1\}^r$ such that $\comp_k(\hat{W}S_3)$ matches $R$ in sign.
\end{enumerate}
\begin{rem}
\label{rem:time-signconsistency}
Computing $\comp_k(\hat{V})$ and $\comp_k(\hat{W})$ respectively in step $1$ and $4$, has time complexity \linebreak $\bigO{k^3\binom{p}{k}\binom{r}{k}}$, analogously to point $1$ of \cref{rem:time-OrderCompSval}. Step $2,3$ have time complexity $\bigO{\binom{r}{k}}$. In step $4$, an exhaustive search over all possible sign permutations assigned to the columns of $\hat{W}\in\bR^{m\times r}$, while checking for a sign match between the columns of its compound and that of $R$, is of $\bigO{\binom{r}{k} 2^r}$ time complexity. Together, \cref{alg:signadjustsvd} results in overall time complexity of $\bigO{\binom{p}{k}\binom{r}{k}k^3+\binom{r}{k} 2^r}$.
\end{rem}

\begin{algorithm}[h]
    \small
    \caption{\small{Sign Adjustment between left ($\hat{V}$) and right ($\hat{W}$}) singular vectors}
    \label{alg:signadjustsvd}
    \begin{algorithmic}[1]
        \State \textbf{Input}: $\tilde{V}, \hat{W}, L, R, r, k$
        \State \textbf{Output}: $\tilde{W}$
        \State Set indices $J = \left\{i\cond L_{:,i}=-[\comp_k(\hat{V}]_{:,i}\ ,\ 1\leq i\leq \binom{r}{k} \right\}$
        \For{$j \gets 1\cdots\binom{r}{k}$}
            \begin{equation*}
            \label{eq:signadjW}
            R_{:,j} = 
            \begin{cases}
            -R_{:,j} & \text{if } j\in J\\
            R_{:,j} & \text{otherwise}
            \end{cases}
            \end{equation*}
        \EndFor
        \State Search for a diagonal $S$ with $s_{ii}\in\{1,-1\}$ for all $i$, such that $\tilde{W}= \hat{W} S$ and $\comp_k(\tilde{W})=R$.
        \State \Return $\tilde{W}$
    \end{algorithmic}
\end{algorithm}
For \cref{ex:examplematrix}, \cref{alg:signadjustsvd} returns the sign-adjusted singular vectors
\begin{equation*}
    \tilde{V} = \hat{V} = 
    \begin{pmatrix*}[r]
    0.58 & -0.60 & -0.34 \\
    0.13 & -0.32 & 0.93 \\
    -0.79 & -0.35 & -0.08 \\
    0.18 & 0.64 & 0.10
    \end{pmatrix*} \quad \text{and} \quad
    \tilde{W} =
    \begin{pmatrix*}[r]
    0.19 & 0.97 & -0.16 \\
    0.33 & 0.02 & 0.47 \\
    0.90 & -0.16 & 0.08 \\
    0.23 & -0.19 & -0.86
    \end{pmatrix*}
\end{equation*}
where in \cref{eq:VrWr}, wedge decomposition gave us,
\[
\hat{W} = 
\begin{pmatrix*}[r]
0.19 & 0.97 & 0.16 \\
0.33 & 0.02 & -0.47 \\
0.90 & -0.16 & -0.08 \\
0.23 & -0.19 & 0.86
\end{pmatrix*}
\]
Thus, by \cref{alg:signadjustsvd} we retrieve a sign-consistent version of the left and right singular vectors of the base matrix $A$.
Next, we present a method for recovering the singular values of $A$ from their compound.

%% file: esvalue.tex
\subsection{Recovering Singular Values}
\label{sec:ESvalueRecover}

In the previous \Cref{sec:restoreMap}, we have established a lexicographic ordering over the singular values of $M=\comp_k(A)$ where $A=V\Sigma W^\top$ is a reduced SVD, and the ordering is given by $\diag(\comp_k(\hat\Sigma))$. Since~$\comp_k(\hat\Sigma)$ corresponds to a reduced SVD, it is invertible with all diagonal entries of~$\hat\Sigma$ being positive.
The following theorem states that we can uniquely recover $\hat\Sigma$ from given~$\comp_k(\hat\Sigma)$.
\begin{lem}
\label{thm:uniqueeval}
    Let~$\cD,\cD' \in \bR^{n \times n}$ be diagonal matrices whose diagonal entries are all positive and $k < n$. If~$\comp_k(\cD) = \comp_k(\cD')$, then~$\cD' = \cD$.
\end{lem}
\begin{proof}
    Let $\cD=\diag([d_1,\ldots,d_n])$ with~$d_1,\dots,d_n \in \bR_{>0}$. Suppose for contradiction that there exists a diagonal $\cD'=\diag([d_1',\ldots,d_n'])$ with~$d_1',\dots,d_n' \in \bR_{>0}$ and such that $\comp_k(\cD')=\comp_k(\cD)$. Consider an index set $I = \{i_1,\ldots,i_{k+1}\}$ such that $1\leq i_1<\cdots<i_{k+1}\leq n$ and let $p,q\in I$ with $p<q$ define the tuples $I_p := I\setminus \{p\}$ and $I_q := I\setminus\{q\}$ in $\cI_{n,k}$. Then, comparing the $I_p$th element of $\diag(\comp_k(\cD))$ and $\diag(\comp_k(\cD'))$ gives
    \begin{equation}
    \label{eq:evalkprod1}
    \prod_{j\in I_p}d_j = \prod_{j\in I_p}d_j'
    \end{equation}
    and similarly for $I_q$th elements
    \begin{equation}
    \label{eq:evalkprod2}
    \prod_{j\in I_q}d_j = \prod_{j\in I_q}d_j'.
    \end{equation}
    Thus, taking the ratio of both sides of \cref{eq:evalkprod1,eq:evalkprod2}, we get,
    \[
    \frac{d_p}{d_q} = \frac{d'_p}{d'_q} \iff \frac{d_p}{d'_p} = \frac{d_q}{d'_q}
    \]
    As the above relation holds for any arbitrary $p,q\in [n]$, the chain rule implies $ d_i=c \cdot d'_i$ for all $i\in[n]$ and some universal constant $c$.
    It is easy to see from the equality $\comp_k(\cD) = \comp_k(\cD')$ and the fact that the diagonal entries of both matrices are all positive that $c=1$.
    This proves the result by contradiction to our initial assumption.
\end{proof}
The following \cref{thm:linearsystem} shows that finding the singular values of $A$ from the singular values of~$\comp_k(A)$ can be achieved by solving a linear system of equations.
\begin{lem}
\label{thm:linearsystem}
    Given $\comp_k(\Lambda)$ for a positive diagonal matrix $\Lambda\in\bR^{r\times r}$ and $k<r$, the diagonal entries of $\Lambda$ are the exponential of the unique solution to the linear system $y = Lx$, where $y=\log(\diag(\comp_k(\Lambda)))$ and $L\in\bR^{\binom{r}{k}\times r}$ is defined as
    \[
    L(I,j) := 
    \begin{cases}
        1 & \text{if }j\in I\\
        0 & \text{otherwise}
    \end{cases}
    \]
    for all $I\in \cI_{r,k}$ and $j\in[r]$.
\end{lem}
\begin{proof} By item 3 of~\cref{thm:compoundProperty}, the diagonal elements of $\comp_k(\Lambda)$ are given by
\[
\lambda_I = \prod_{j\in I} \lambda_j
\]
or equivalently, 
\begin{equation}
\label{eq:loglinear}
    \log \lambda_I = \sum_{j\in I} \log \lambda_j
\end{equation}
for each of the indices $I\in\cI_{r,k}$. The above relation gives us the following linear system to be solved for $x\in\bR_{>0}^{r}$:
\begin{equation}
\label{eq:linearSystem}
    y = Lx
\end{equation}
where $y=\log(\diag(\comp_k(\Lambda)))\in\bR^{\binom{r}{k}}$ and the mapping $L:\bR^r \to \bR^{\binom{r}{k}}$ is defined as,
\begin{equation*}
    L(I,j) := 
    \begin{cases}
        1 & \text{if }j\in I\\
        0 & \text{otherwise}
    \end{cases}
\end{equation*}
each row corresponding to the composition of $\lambda_I$ for $I\in\cI_{r,k}$, as per \cref{eq:loglinear}. For any solution $x^\ast$ of the linear system, $\Lambda = \diag(\exp(x^\ast))$ gives the diagonal matrix.
\end{proof}
In \cref{ex:examplematrix}, the solution to $\hat\Sigma$ with $r=3$ can be obtained by solving
\[
    \begin{pmatrix}
    4.38 \\
    4.13 \\
    3.81
    \end{pmatrix}
    =
    \begin{pmatrix}
    1 & 1 & 0 \\
    1 & 0 & 1 \\
    0 & 1 & 1
    \end{pmatrix}
    x
\]
which gives $\diag(\hat\Sigma)=\exp(x)=[10.50,7.60,5.92]^\top$, and by that recovers the true values.

%% file: proof_main.tex
\section{Proof of \cref{thm:main}}
\label{sec:proof_main}
Now that we have discussed all parts involved in the computation of an inverse-compound for a compound-decomposable matrix $M\in \bR^{\binom{n}{k}\times \binom{m}{k}}$ and given $k\leq \min\{n,m\}$, we can logically stitch them together to arrive at \cref{thm:main}. Assume that $A\in\bR^{n\times m}$ has $\rank(A)=r$, $M = \comp_k(A)$ with compact SVD $M = L S R^\top$, where $L\in \bR^{\binom{n}{k}\times \binom{r}{k}}, R\in \bR^{\binom{m}{k}\times \binom{r}{k}}$ and $S\in \bR^{\binom{r}{k}\times \binom{r}{k}}$. Further, let the (unknown) $A$ have a compact SVD $A=V\Sigma W^\top$. For $k<r$, the inverse-compound method consists of two components: i) find $V, W$ from $L, R$ respectively, and ii) find $\Sigma$ from $S$.  
\cref{ass:svdsval} and \cref{prop:distinct-sval-svd} ensure that $L$ and $R$ coincide with $\comp_k(V)$ and $\comp_k(W)$ respectively, up to a fixed column permutation, and simultaneous reversals of column signs.

Given $L, R$ and $k<r$, \cref{thm:unique-decomposition} provides the recovery of $\hat{V}$ and $\hat{W}$, which coincide with $V$ and $W$, respectively, up to sign reversals of the columns.
Although the columns of $\hat{V}$ and $\hat{W}$ are order-consistent, \cref{prop:sign-inconsistency-svd} shows possible sign inconsistency between them, which is then restored by \cref{alg:signadjustsvd}, giving $\tilde{W}$ with column signs consistent with that of $\tilde{V}=\hat{V}$, as guaranteed by \cref{thm:signconsistency}.
\Cref{sec:restoreMap} restores a correct lexicographical order over the diagonal entries of $S$, given by $\comp_k(\hat{\Sigma})$, where $\diag(\hat{\Sigma})$ is a permuted version of $\diag(\Sigma)$, which is order-consistent with $\hat{V}$. The reordered $\diag(S)$, i.e., $\diag(\comp_k(\hat{\Sigma}))$, while given as the input to a linear system in \cref{thm:linearsystem}, outputs the solution $\hat\Sigma$, due to the uniqueness guarantee in \cref{thm:uniqueeval}.
Lastly, by composition, we obtain $\tilde{V} \hat\Sigma \tilde{W}^\top = cA$, where $c \in \{1,-1\}$ and $c^k=1$, thus proving \cref{thm:main}.

The total time complexity of the inverse-compound method, in the case of $k<r$, is determined by the singular vector recovery component, i.e., by aggregating the time complexity results presented in \cref{thm:unique-decomposition}, \cref{rem:time-OrderCompSval}, and \cref{rem:time-signconsistency} respectively.
However, for $k=r$, the time complexity is only determined by \cref{thm:unique-decomposition}.
The time complexity for the singular value recovery component is linear in $r\binom{r}{k}$, hence subsumed by the previous components.
Note that the time complexity provided in the theorem implicitly takes into account the SVD (compact) computation time of $M$ with rank $\binom{r}{k}$, which is $\bigO{\binom{n}{k}\binom{m}{k}\binom{r}{k}}$ (see, e.g., \cite{trefethen1997numerical}).

For $k=\rank(A)<\min\{n,m\}$, i.e., $\rank(\comp_k(A)) = 1$ by item $4$ of \cref{thm:compoundProperty}, we proved non-uniqueness of inverse-compounds in \cref{prop:compound_nonunique}.
Our algorithm to recover the set of inverse-compounds for such a case is then as follows:
let~$M \in \bR^{\binom{n}{k} \times \binom{m}{k}}$ be a compound-decomposable matrix of rank one. Compute some compact SVD of~$M$, i.e., $\sigma > 0, u \in \bR^{\binom{n}{k}}, v \in \bR^{\binom{m}{k}}$ such that~$M = \sigma uv^\top$ and construct $M_u,M_v$ as in \cref{thm:wedgeMat}. Then, compute $U \in \bR^{n \times k}, V \in \bR^{m \times k}$ to be some orthonormal basis matrices of the kernels of~$M_u,M_v$ respectively, i.e., such that~$U^\top U = V^\top V = I_k$ and~$\spanop(U) = \ker(M_u)$ and~$\spanop(V) = \ker(M_v)$. By construction, the~$k$-dimensional subspaces associated with the (decomposable and nonzero) vectors~$\comp_k(U),\comp_k(V)$. Hence, there exist positive scalars~$\alpha,\beta$ such that~$u = \alpha \comp_k(U)$ and~$v = \beta \comp_k(V)$. Let~$\Sigma \in \bR^{k \times k}$ be the diagonal matrix~$\Sigma = (\alpha \sigma \beta)^{1/k} I_k$, then $A = U \Sigma V^\top$ is such that~$\comp_k(A) = M$. To construct the whole set in \cref{prop:compound_nonunique}, note that~$A = (U \Sigma) V^\top$ is a rank-revealing factorization of~$A$, such that the whole pre-image is given by
\[
    \{ B \in \bR^{n \times m}: U\Sigma T V^\top \text{ with } \det(T) = 1\}.
\]

%% file: conclusion.tex
\section{Conclusion}
\label{sec:conclusion}
In this paper, we presented an algorithm, which, given a compound-decomposable matrix~$M \in \bR^{\binom{n}{k} \times \binom{m}{k}}$, provably finds a matrix~$A\in\bR^{n\times m}$ such that the $k$th multiplicative compound matrix of $A$ equals $M$. The algorithm is based on singular value decomposition of $M$, and its time complexity is polynomial in the dimensions of~$M$.

There are several directions for future work. First, the application of this algorithm to questions in dynamical systems and control theory will be the focus of follow-up work. Second, the question of whether a given matrix is compound-decomposable remains open. Finally, assuming a given matrix is not compound-decomposable, it might still be useful to find matrices whose $k$th multiplicative compound is the best possible approximation, and designing an algorithm to find such matrices will be done in future work.

\section*{Acknowledgements}
The work of Debojyoti Dey and Christian Grussler was supported by the Israel Science Foundation (grant no. 2406/22) and the Bernard M. Gordon Center for Systems Engineering at the Technion -- IIT. The work of Ron Ofir was supported in part by the Air Force Office of Scientific Research, under award numbers FA9550-23-1-0175 and FA9550-25-1-0223, and by the Viterbi Fellowship, Technion -- IIT.

%% file: a1.tex
\section{Proof of \cref{thm:adjcomp_bijection}}
\crefalias{section}{appendix} %
\label{app:proofs}
    \noindent We recall that the adjugate $\adj(A)$ of $A\in\bR^{n\times n}$ is given by
    \[
    [\adj(A)]_{ij} = (-1)^{i+j}M_{ji}
    \]
    where $M_{ij}$ is the determinant of the matrix that results from $A$ by deleting its $i$th row and $j$th column, i.e., $M_{ij}=\det(A_{[n]-\{i\},[n]-\{j\}})$.
    From the definition of multiplicative compound, this means that
    \[
    [\comp_{n-1}(A)]_{ij} = M_{n-i,n-j}.
    \]
    Thus, $\adj(\cdot)$ and $\comp_{n-1}(\cdot)$ are related by
    \begin{equation}
    \label{eq:adjcompmap}
        [\adj(A) ^\top]_{ij} = (-1)^{i+j}[\comp_{n-1}(A)]_{n-i,n-j}
    \end{equation}
    which leads to the transformation formula
    \begin{equation}
    \label{eq:adj-comp-transformation_1}
    \adj(A) = S P \comp_{n-1}(A)^\top P S,
    \end{equation}
    where $P=(p_{ij})\in \{0,1\}^{n\times n}$ is symmetric and anti-diagonal with non-zero entries $p_{i,n-i+1}=1$ such that $[P X P]_{ij}=X_{n-i,n-j}$, and $S=(s_{ij})\in\{-1,0,1\}^{n\times n}$ is diagonal with $s_{ii}=(-1)^i$ such that $[SXS]_{ij}= (-1)^{i+j}X_{ij}$ for an $X\in\bR^{n\times n}$.
    In the following, we summarize some useful properties of $S$ and $P$:
    \begin{enumerate}
        \item $\det(S) = (-1)^{\frac{n(n+1)}{2}}$.
        \item Since $
        [\comp_{n-1}(S)]_{ii} =
        \det(S) (-1)^{n-i+1} = (-1)^ {\frac{n(n+3)}{2}-i+1}
        $, $\comp_{n-1}(S) = (-1)^{\frac{n(n+3)}{2}+1} S$.
        \item $\det(P) = (-1)^{\frac{(n-1)n}{2}}$ and, $\comp_{n-1}(P) = (-1)^{\frac{(n-1)(n-2)}{2}} P$.
        \item $T:= SP$ is anti-diagonal with non-zero entries $t_{i,(n-i+1)}=(-1)^i$.
        \item $SP = (PS)^\top$
        \item $(SP)^2 = (-1)^{n+1} I_n = (PS)^2$.
    \end{enumerate}
    Using these properties, and substituting $A$ by $\adj(A)$ in \cref{eq:adj-comp-transformation_1}, gives
    \begin{align*}
        \adj(\adj(A)) &= S P \comp_{n-1}(S P \comp_{n-1}(A)^\top P S)^\top P S\\
        &= (SP)^2 \comp_{n-1}(\comp_{n-1}(A)) (PS)^2 &&(\text{by Cauchy-Binet and  point }2,3)\\
        &= \comp_{n-1}(\comp_{n-1}(A)). &&(\text{by point }6).
    \end{align*}
    This completes the proof of \cref{thm:adjcomp_bijection}. \qed